\documentclass[12pt]{article}
\usepackage{amsfonts}
\usepackage{latexsym}
\usepackage{color, amsmath,amssymb, amsfonts, amstext,amsthm, latexsym, dsfont}
\usepackage{epic,eepic}
\usepackage{graphicx}
\usepackage{epstopdf}
\usepackage{longtable}
\usepackage[title]{appendix}
\usepackage{indentfirst}
\usepackage{graphicx}
\usepackage{float}
\usepackage{graphicx}
\usepackage{subfig}
\usepackage{url}
\usepackage[colorlinks=false, linktocpage=true]{hyperref}

\newcommand{\diag}{{\text{\upshape diag}}}
\newtheorem{Defi}{Definition}

\newtheorem{thm}{\bf Theorem} [section]
\newtheorem{lem}[thm]{Lemma}

\newtheorem{rem}[thm]{Remark}
\numberwithin{equation}{section}
 \numberwithin{Lem}{section}
 \numberwithin{Defi}{section}
 \numberwithin{Theo}{section}
 \numberwithin{Rem}{section}
  \numberwithin{Coro}{section}
  \numberwithin{Fig}{section}

\title{Transition pathways for a class of high dimensional  stochastic dynamical systems with L\'evy noise
\footnotetext{\\
Jianyu Hu\\
School of Mathematics and Statistics \& Center for Mathematical Sciences \& Hubei National Center for Applied Mathematics, Huazhong University of Science and Technology, Wuhan
430074, China.\\
E-mail: jianyuhu@hust.edu.cn}
\footnotetext{\\Jianyu Chen\\
School of Mathematics and Statistics \& Center for Mathematical Sciences \& Hubei National Center for Applied Mathematics, Huazhong University of Science and Technology, Wuhan
430074, China.\\
E-mail: jianyuchen@hust.edu.cn\\
$^*$ Corresponding author}}
\author{Jianyu Hu and Jianyu Chen$^{*}$}

\begin{document}

\maketitle
\begin{abstract}
This work is devoted to deriving the Onsager-Machlup action functional for a class of stochastic differential equations with (non-Gaussian) L\'{e}vy process as well as Brownian motion in high dimensions. This is achieved by applying the Girsanov transformation for probability measures and then by a path representation. The Poincar\'{e} lemma is essential to handle such path representation problem in high dimensions. We provide a sufficient condition on the vector field such that this path representation holds in high dimensions. Moreover, this Onsager-Machlup action functional may be considered as the integral of a Lagrangian. Finally, by a variational principle, we investigate the most probable transition pathways analytically and numerically.


\end{abstract}

\section{Introduction}
Stochastic dynamical systems are mathematical models for complex phenomena in physical, chemical and biological sciences \cite{arnold2013random,duan2015introduction,duan2014effective,imkeller2012stochastic}. In contrast to the deterministic systems, noisy fluctuations result in the possibility of transitions between metastable states \cite{ditlevsen1999observation,qiu2000kuroshio}.  For many irreversible systems, the absence of detailed balance results in the difficulties of analyzing the asymptotic behaviors. Onsager-Machlup action functional is an essential tool to study these transition phenomena. We will derive the Onsager-Machlup action functional for a class of stochastic differential equations with non-Gaussian L\'{e}vy process in high dimensions and this enables to study the most probable transition pathway for many scientific models.

Onsager and Machlup considered the probability of paths of a diffusion process \cite{onsager1953fluctuations}. For stochastic differential equations, the Onsager-Machlup action functional has been widely investigated \cite{capitaine1995onsager,capitaine2000onsager,fujita1982onsager,hara2016stochastic,ikeda2014stochastic,shepp1992note,takahashi1981probability}. The key point is to express the transition probability of a diffusion process in terms of a functional integral over paths of the process. As in D\"{u}rr and Bach \cite{durr1978onsager}, Onsager-Machlup action functional may be regarded as a Lagrangian for determining the most probable path of a diffusion process by a variational principle.

Note that the most existing works on Onsager-Machlup action functional are for systems with Brownian motion. However, non-Gaussian random phenomena including heavy-tailed distributions have been found to be suitable in modeling biological evolution, climate change and other complex scientific and engineering systems \cite{bottcher2013levy,hao2014asymmetric,jourdain2012levy,zheng2016transitions}. So we are urgent to derive the Onsager-Machlup action functional for stochastic dynamical systems with L\'{e}vy noise. Bardina et.al \cite{bardina2002asymptotic} tried to deal with jump functions directly rather than using the Girsanov  theorem to obtain the asymptotic evaluation of the Poisson measure for a tube in the path space. More recently, Chao and Duan \cite{chao2019onsager} used the Girsanov transformation to absorb the vector field or drift term and thus derived the Onsager-Machlup action functional in one dimensional systems with addictive L\'{e}vy process of small jumps. But this result does not apply to general high dimensional stochastic differential equations. Hu and Duan \cite{hu2020onsager} derived the Onsager-Machlup action functional for a stochastic partial differential equations with additive noise, but did not have analysis for Euler-Lagrange equations and the most probable transition pathways.

In this paper, we derive the Onsager-Machlup action functional for a class of stochastic differential equations with additive L\'{e}vy noise in high dimensions, assuming the vector field has a certain structure. We use the Brownian motion to absorb the drift vector field by applying the Girsanov transformation for probability measures and then achieve the path representation by Poincar\'{e} lemma for high dimensions.

The Onsager-Machlup action functional could be considered as the integral of a Lagrangian. By a variational principle, the most probable pathway connecting two metastable states satisfies either the Euler-Lagrange equation or the corresponding Hamiltonian system. Thus, it reduces to a two-point boundary value problem when we capture the most probable pathway. The shooting method can further analyze such a path numerically.

This article is arranged as follows. In section 2, we derive the Onsager-Machlup action functional for a class of stochastic differential equations with additive L\'{e}vy noise in high dimensions. In section 3, we introduce the Lagrangian mechanics for studying the most probable transition pathways. In section 4, we implement our theoretical approaches on the Maier-Stein systems. Finally, some discussions are in section 5.

\section{Onsager-Machlup action functional}
As scientific modeling requries taking L\'evy noise into account, and the L\'evy-Ito decomposition theorem \cite{applebaum2009levy} tells us that a large class of noise is the sum of Brownian noise and L\'evy noise, we thus consider the following stochastic differential equation:
\begin{equation}
\begin{split}
dX(t)=f(X(t))dt+g(X(t))dW(t)+ dL(t), t\in[0,1],
\end{split}
\end{equation}
with initial data $X(0)=x_0 \in \mathbb{R}^d$, where $f:\mathbb{R}^d\rightarrow\mathbb{R}^d$ is a regular function, $g$ is a $d\times d$ matrix-valued function in $\mathbb{R}^d$, $W$ is a Brownian motin in $\mathbb{R}^d$, and $L$ is a L\'{e}vy process in $\mathbb{R}^d$ which has the following L\'{e}vy-It\^{o} formulation
\begin{equation}
\begin{split}
L(t)=\sum\limits_{j=1}^{d}L_j(t)=\sum\limits_{j=1}^{d}\int_{|z|<1}z\tilde{N}_j(t,dz)e_j+\sum\limits_{j=1}^{d}\int_{|z|\geq1}zN_j(t,dz)e_j.
\end{split}
\end{equation}
Here, the basis $\{e_j\}_{j=1}^d$ is orthonormal in $\mathbb{R}^d$, $N_j(dt,dz)$ is the Poisson measure associated with one-dimensional L\'{e}vy process $L_j(t)$, and $\tilde{N}_j(dt,dz)=N_j(dt,dz)-\nu_j(dz)dt$ is the corresponding compensated Poisson random measure. We will also denote by $X^L$ the stochastic convolution, that is, the solution to (2.1) with $f = 0$ and $x_0 = 0$.

\subsection{The framework}
The well-posedness of  stochastic differential equation (2.1) has been widely investigated. The following theorem is a special case of \cite[Theorem 3.1]{albeverio2010existence}.
\begin{thm}(Global Solution)
Assume that the functions f and g are locally Lipschitz and satisfy “one sided linear growth” condition in the following sense:

(1) (Locally Lipschitz) For every $K>0$ there exists a constant $C_1>0$ such that, for all $|x|,|y|\leq K$,
\begin{align*}
 |f(x)-f(y)|^2+|g(x)-g(y)|^2\leq C_1|x-y|^2.
\end{align*}

(2) (One sided linear growth) There exists a constant $C_2>0$ such that, for all $x\in \mathbb{R}^d$,
\begin{align*}
tr(g(x)^{T}g(x))+2x\dot f(x)\leq  C_2|x|^2.
\end{align*}

Then, there exists a unique global solution to (2.1) and the solution process is adapted and c\`{a}dl\`{a}g.
\end{thm}

Our goal is to study the limiting behaviour of ratios of the form
\begin{equation}
\begin{split}
\gamma_{\epsilon}(\phi)=\frac{\mathbb{P}(\|X-\phi\|\leq\epsilon)}{\mathbb{P}(\|X^L\|\leq\epsilon)},
\end{split}
\end{equation}
as $\epsilon$ tends to $0$. Here $\phi$ is a deterministic function satisfying some regularity conditions and $\|\cdot\|$ is a suitable norm defined on the functions from $[0, 1]$ to $\mathbb{R}^d$.

Denote the space of the solution paths of (2.1) by $D^{x_0}([0,1])$, which is defined as
\begin{align*}
D^{x_0}([0,1])=\{x(t)|x:[0,1]\rightarrow \mathbb{R}^d, x\ is\ c{a}dl{a}g, x(0)=x_0\}.
\end{align*}
As every c\`{a}dl\`{a}g function on $[0,1]$ is bounded, we shall define the uniform norm
\begin{equation}
\begin{split}
\|x\|=\sup\limits_{0\leq t \leq 1}|x(t)|\ ,x\in D^{x_0}([0,1]).
\end{split}
\end{equation}
Equipped with this uniform norm, $D^{x_0}([0,1])$ is a Banach space. Now, we define the Onsager-Machlup action functional \cite{bardina2003onsager} associated with the stochastic differential equation (2.1).
\begin{Defi}
When $\lim\limits_{\epsilon\rightarrow0}\gamma_{\epsilon}(\phi)=\exp\{-J_0(\phi)\}$ for all $\phi$ which are located in a reasonable class of functions, the functional $J_0$ is called the Onsager-Machlup action functional associated to system (2.1). The process $X^L$ is called the reference process.
\end{Defi}
The Karhunen–Lo\`{e}ve expansion \cite{bardina2003onsager} method is applied to investigate the Onsager-Machlup action functional by converting stochastic integrals to independent random variables, which is invalid for L\'evy-driven process due to the nonequivalence of irrelevance and independence. It is hard to analyze the limiting behaviors. But the path representation method works well here, because it treats the jumps as a whole by It\^{o}'s formula and then we are able to control these parts through the following lemma.
\begin{lem}
A L\'{e}vy process with L\'{e}vy triplet $(A,\nu, b)$ has bounded variation if and only if $A=0$ and the jump measure satisfies $\int_{|x|<1}x\nu(dx)<\infty$.
\end{lem}
We refer to \cite[Page 112]{applebaum2009levy}  \cite[Page 15]{bertoin1996levy} for the detail proof.

The Onsager-Machlup action functional quantifies the probability of the solution path up to a given function on a small tube. If we minimize this functional given two metastable states, we will obtain the most probable transition pathway between these two states \cite{ren2020identifying,zheng2020maximum}. This is an important application to study the dynamical behaviours for stochastic differential equations.

\subsection{Derivation of Onsager-Machlup action functional}
In this subsection, we present the derivation of the Onsager-Machlup action functional for a class of high dimensional stochastic dynamical systems. The reference process $X^L$ is effective if the probability $\mathbb{P}(\|X^L\|\leq\epsilon)$ is not related to the drift coefficient and diffusion coefficient. Under the uniform norm (2.4), no such effective reference process exists in the multiplicative case. See more discussions in Section 5. To this end, we shall consider the stochastic differential equation with additive noise:
\begin{equation}
\begin{split}
dX(t)=f(X(t))dt+BdW(t)+ dL(t), t\in[0,1],
\end{split}
\end{equation}
with initial data $X(0)=x_0 \in \mathbb{R}^d$, where $B$ is a nondegenerate $d\times d$ matrix.

We suppose that our drift coefficient is locally Lipschitz and satisfy "one sided linear growth" condition. So by Theorem 2.1, there exists a unique global solution to (2.5) and the solution process is adapted and c\`{a}dl\`{a}g.

Before deriving the Onsager-Machlup action functional, we first introduce an auxiliary system and take our deterministic functions $\phi$ from a so-called Cameron-Martin space. Fix a function $h\in L^2([0,1];\mathbb{R}^d)$. Let $\phi^h$ be the solution of the finite dimensional equation
\begin{equation}
\begin{split}
&d\phi^h(t)=Bh(t)dt, t\in[0,1],\\
&\phi^h(0)=x_0\in \mathbb{R}^d.
\end{split}
\end{equation}
Our result is of high dimensions, so we need the following lemma, which ensures us to generalize the result from one dimension to high dimensions.
\begin{lem}
Let $F$ be a mapping from $\mathbb{R}^d$ to $\mathbb{R}^d$. For each $x\in \mathbb{R}^d$, assume that $DF(x)$ is a symmetric matrix from $\mathbb{R}^d$ to $\mathbb{R}^d$. Then, there exists a smooth function $V:\mathbb{R}^d\rightarrow \mathbb{R}$, such that for all $x\in \mathbb{R}^d$, DV(x)=F(x).
\end{lem}

\begin{proof}
The dual space of $\mathbb{R}^d$ is itself. Also, the tangent bundle and cotangent bundle are both $\mathbb{R}^{2d}$. So for each $x\in\mathbb{R}^d$, $F(x)$ can be regard as a 1-form. Then it is a closed 1-form due to the symmetry of $DF(x)$. Thus by the Poincar\'{e} lemma \cite[Page 137 Theorem 4.1]{lang2012fundamentals} \cite[Page 350 Theorem 33.20]{kriegl1997convenient}, it is an exact form, i.e. there exists a smooth function $V:\mathbb{R}^d\rightarrow \mathbb{R}$, such that for all $x\in \mathbb{R}^d$, $DV(x)=F(x)$. $\square$
\end{proof}

\begin{rem}
Since we try to use It\^{o}'s formula to represent the stochastic integral, it requires the existence of the primary function for the drift coefficient. For one dimension \cite{chao2019onsager}, it is clear that the primary function is just its potential function. But in high dimensions, the primary function does not always exist. We need the symmetry of the drift coefficient to ensure the existence of the primary function. The Poincar\'{e} lemma provides a sufficient condition for the existence of the primary function, which states the relation between the closed form and the exact form.
\end{rem}

Now, we present a theorem about Onsager-Machlup action functional.
\begin{thm} $\mathbf{(Onsager-Machlup\ action\ functional)}$
Assume that the matrix $B$ is nondegenarate such that $B^{-1}f$ is $C_b^2$ in $x$, $\phi^h$ is defined by (2.6), and the L\'{e}vy jump measure $\nu$ satisfies that $\int_{|x|<1}|x|\nu(dx)<\infty$. Let $g(x)=(B^{-1})^{*}(B^{-1}f(x))$. If the gradient $\nabla_x g(x)$ is symmetric, the Onsager-Machlup action functional of system (2.5) is given by
\begin{equation}
\begin{split}
J_0(\phi^h,\dot{\phi}^h)=\int_0^1L(\phi^h,\dot{\phi}^h)ds,
\end{split}
\end{equation}
where the Lagrangian is
\begin{equation}
\begin{split}
L(\phi^h,\dot{\phi}^h)=\frac{1}{2}|B^{-1}[f(\phi^h(t))-\dot{\phi}^h(t)-\eta]|^2+\frac{1}{2}\mbox{Tr}[\nabla_xf(\phi^h(s))],
\end{split}
\end{equation}
with $\eta=\int_{|\xi|<1}\xi\nu(d\xi)$.
\end{thm}
\begin{proof}
The detailed proof is in Appendix A.
\end{proof}
\begin{rem}
In Theorem 2.5, we see that, the quadratic term is the main term, while the divergence term comes from the It\^{o} correction of Brownian motion. Moreover, only small jumps contribute to the Onsager-Machlup action functional and the effect is similar to adding the mean value of small jumps to the drift. Furthermore, if the integral $\eta=\int_{|\xi|_{\mathbb{R}^d}<1}\xi\nu(d\xi)$ is the zero element in $\mathbb{R}^d$, the L\'{e}vy noise will have no effect on the Onsager-Machlup action functional.
\end{rem}

\begin{rem}
In our result, we require the symmetry of the gradient $\nabla_x g(x)$, while in one dimension \cite{chao2019onsager}, the drift does not. In high dimension, we apply the Poincar\'{e} lemma which requires the symmetry condition to obtain the original function. Actually in one dimension, the symmetry condition already holds. Furthermore, we could deal with the large jumps in our models because they can be controlled by the bounded variation. So for the $\alpha$-stable L\'{e}vy motion, the corresponding jump measure is given by
\begin{align*}
\nu(dx)=c_1|\xi|^{-1-\alpha}\chi_{(0,\infty)}d\xi+c_2|\xi|^{-1-\alpha}\chi_{(-\infty,0)}d\xi.
\end{align*}
Thus, $\int_{|x|<1}|x|\nu(dx)<\infty$ if and only if $\alpha<1$. Therefore, Theorem 2.5 is adaptive to the $\alpha$-stable L\'{e}vy motion with $\alpha<1$.
\end{rem}
\begin{rem}
Note that $g(x)=(B^{-1})^{*}(B^{-1}f(x))$. If the matrix $B$ is symmetric, we need the symmetry of the gradient $\nabla_x f(x)$. Thus in this case, Theorem 2.5 holds only for potential systems. If $B$ is not a  square matrix, this theorem still holds since we can deal with the inverse for $BB^{*}$.
\end{rem}

\section{Euler-Lagrange equation and the most probable transition pathway}
In this section, we restrict ourselves on transition pathways with fixed initial and final states. We directly seek the minimizer of the Onsager-Machlup action functional $J_0(\phi^h,\dot{\phi}^h)$ to find approximation of the most probable transition pathway whenever it exists. We emphasize that, this most probable transition pathway is not the real path for the stochastic system, but it captures particle paths of the largest probability around its neighbourhood. As in classical mechanics, if the minimizer exists in classical sense, it satisfies the corresponding Euler-Lagrange equation.

Denote the Onsager-Machlup action functional $I[z]=J_0(z,\dot{z})=\int_0^1L(z,\dot{z})dt$ with Lagrangian given by Theorem 2.5
\begin{equation}
\begin{split}
L(z,\dot{z})=\frac{1}{2}|B^{-1}[f(z(t))-\dot{z}(t)-\eta]|^2+\frac{1}{2}Tr[\nabla_xf(z(t))],
\end{split}
\end{equation}
where $\eta=\int_{|\xi|_{\mathbb{R}^d}<1}\xi\nu(d\xi)$. Denote $\mathcal{A}=\{z\in H^1([0,1];\mathbb{R}^d),z(0)=z_0,z(1)=z_1\}$, the admissible set consisting of transition paths and $C_B$ the matrix norm of $B$.
\begin{thm}$\mathbf{(Euler-Lagrange\ Equation)}$ Assume that the drift term $f$ is $C^1$, which is global Lipschitz and its derivative is bounded below, i.e. there exists a constant $K$ such that  $\partial_if^i(z)\geq K$ for $i=1,...,d$. Also, we suppose that the jump measure $\nu$ satisfies $\frac{1}{d}\sum_{i=1}^d|f_i(x_0)-\eta_i|^2\geq\frac{C_2}{C_1}K$ for two positive constants $C_1$ and $C_2$ depending on Lipschitz constants of $f$. Then

(i) there exists at least one function $z^{*}$ such that $I[z^{*}]=\min\limits_{z\in\mathcal{A}}I[z]$. Furthermore, if we restrict ourselves to twice differentiable functions $z$, then it satisfies the Euler–Lagrange equation
\begin{equation}
\begin{split}
&\langle B^{-1}[\sum_{j=1}^d\frac{\partial f(z)}{\partial z_j}-\ddot{z}(t)],B^{-1}e_i\rangle\\
=\langle B^{-1}&[f(z(t))-\dot{z}(t)-\eta],B^{-1}\frac{\partial f(z)}{\partial z_i}\rangle+\frac{1}{2}\sum\limits_{j=1}^d\frac{\partial^2 f^j(z)}{\partial z_i\partial z_j},
\end{split}
\end{equation}
for $i=1,...,d$, where $z(0)=z_0,z(1)=z_1$.

(ii) Conversely, the classical solution of the Euler–Lagrange equation (3.2) is a local minimizer of $I[z]$.
\end{thm}
\begin{proof}
According to the proof of \cite[Theorem 5.1]{chao2019onsager} or \cite[Lemma 4.2]{wan2018convergence}, we obtain
\begin{equation}
\begin{split}
I[z]\geq \frac{1}{2C_2C_B^2}\int_0^1|\dot{z}(t)|^2dt-\frac{C_1}{2C_2C_B^2}\sum_{i=1}^d|f_i(x_0)-\eta_i|^2+\frac{1}{2C_B^2}K,
\end{split}
\end{equation}
where $C_1$ and $C_2$ are two positive constants depending on Lipschitz constants of $f$. Since $\frac{1}{d}\sum_{i=1}^d|f_i(x_0)-\eta_i|^2\geq\frac{C_2}{C_1}K$, the coercivity on $I[\cdot]$ follows. On the other hand, $L(z,\dot{z})$ is convex in the variable
$\dot{z}$. By theorem 2 of \cite[Page 448]{evans2010pdes}, we know that there exits at least one minimizer of $I[z]$. Moreover, if it is $C^2$, it satisfies the Euler–Lagrange equation (3.2). Hence, we reach first conclusion. For the second result, since  $L(z,\dot{z})$ is convex in the variable $\dot{z}$, by the theorem 5 of \cite[Page 453]{evans2010pdes}, it is a local minimizer of $I[z]$. This completes the proof of Theorem 3.1.
\end{proof}

\begin{rem}
The existence theory of minimizers for the Onsager-Machlup functional in the space of smooth
functions through direct minimization is incomplete. Thus, Theorem 3.1 offers a sufficient condition to establish a minimizer among the functions in $\mathcal{A}$. The minimizer of $I[\cdot]$ could give the most probable pathway for system (2.5) and also offers a sufficient condition to obtain the minimizer in $C^2(\mathbb{R}^d)$ through the Euler–Lagrange equation.
\end{rem}

\begin{rem} Suppose that $B$ is a diagonal matrix $\diag\{b_1,...,b_d\}$. Since $Df(x)$ is symmetric for each $x\in\mathbb{R}^d$ and  $i=1,...,d$,  the Euler–Lagrange equation (3.2) reduces to the following Newton system
\begin{equation}
\begin{split}
\ddot{z_i}=\sum\limits_{j=1}^d\frac{b_i^2}{b_j^2}(f^j(z)-\eta_j)\frac{\partial f^j(z)}{\partial z_i}+\frac{b_i^2}{2}\sum\limits_{j=1}^d\frac{\partial^2 f^j(z)}{\partial z_i\partial z_j},
\end{split}
\end{equation}
with two boundary conditions
\begin{equation}
\begin{split}
z(0)=z_0,z(1)=z_1.
\end{split}
\end{equation}
\end{rem}

\section{An application to the Maier-Stein system}
Overdamped systems without detailed balance arise in the study of disordered materials, chemical reactions far from equilibrium, and theoretical ecology.
As a typical example of the overdamped systems, the Maier–Stein system has been chosen to illustrate our results. We shall consider the following Maier-Stein system \cite{maier1993escape,maier1996scaling}
\begin{equation}
\begin{split}
dx&=(x-x^3-\gamma xy^2)dt+dW(t)+dL_1(t)\\
dy&=-(1+x^2)ydt+dW(t)+dL_2(t),
\end{split}
\end{equation}
where $\gamma$ is a positive parameter, $W(t)$ is Brownian motion, and $L_1(t),L_2(t)$ are stable L\'{e}vy processes with $\alpha_1,\alpha_2<1$. As in \cite[Page210]{duan2015introduction}, we denote that $L_1 \sim S_{\alpha_1}(\sigma_1,\beta_1,\mu_1)$, and $L_2 \sim S_{\alpha_2}(\sigma_2,\beta_2,\mu_2)$. The characteristic exponent $\alpha$ lies in the range (0,2] and determines
the rate at which the tails of the distribution decrease. The scale
parameter $\sigma$ compresses or extends the distribution about $\mu$. The parameter $\beta$ determines the skewness of the distribution, which lies [-1,1]. The location parameter $\mu$ shifts the distribution to the left or right.

Note that there exist two stable nodes SN1$(-1,0)$ and SN2$(1,0)$ and one unstable node US(0,0) for the corresponding deterministic system. We check the symmetry of $Df$ with $f(x,y)=(x-x^3-\gamma xy^2,-(1+x^2)y)^T$.
\begin{equation}
Df(x,y)=\left(
  \begin{array}{ccc}
    1-3x^2-\gamma y^2 & -2\gamma xy\\
    -2xy & -(1+x^2) \\
  \end{array}
\right)
\end{equation}
\begin{figure}[hp]
    \centering
    \subfloat[The vector field of Maier-Stein system]{
    % \subfigure[]{
    \includegraphics[scale=0.45]{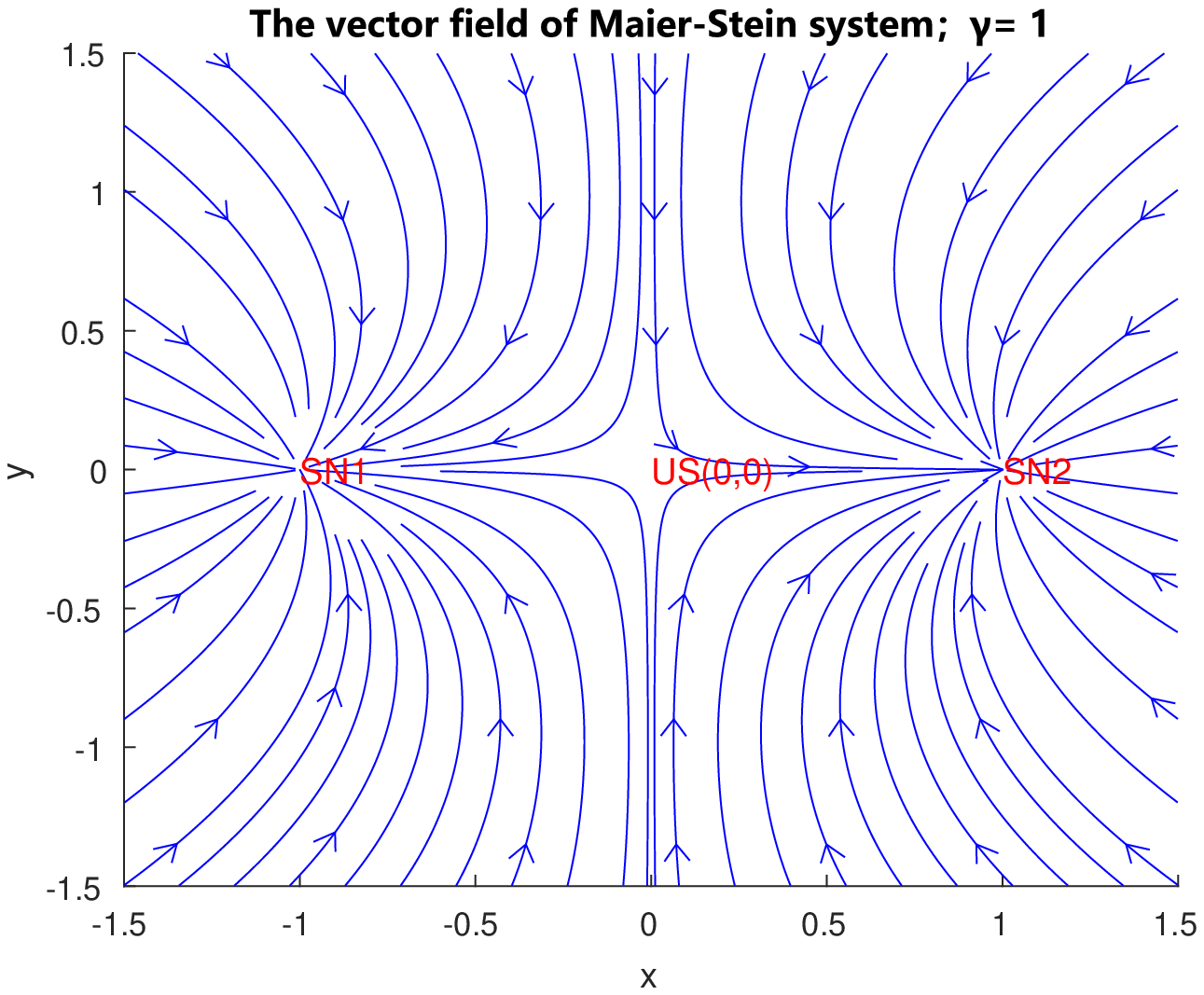}
    \label{fig:fig1vectorfield}
    }
    \subfloat[Potential]{
    \includegraphics[scale=0.35]{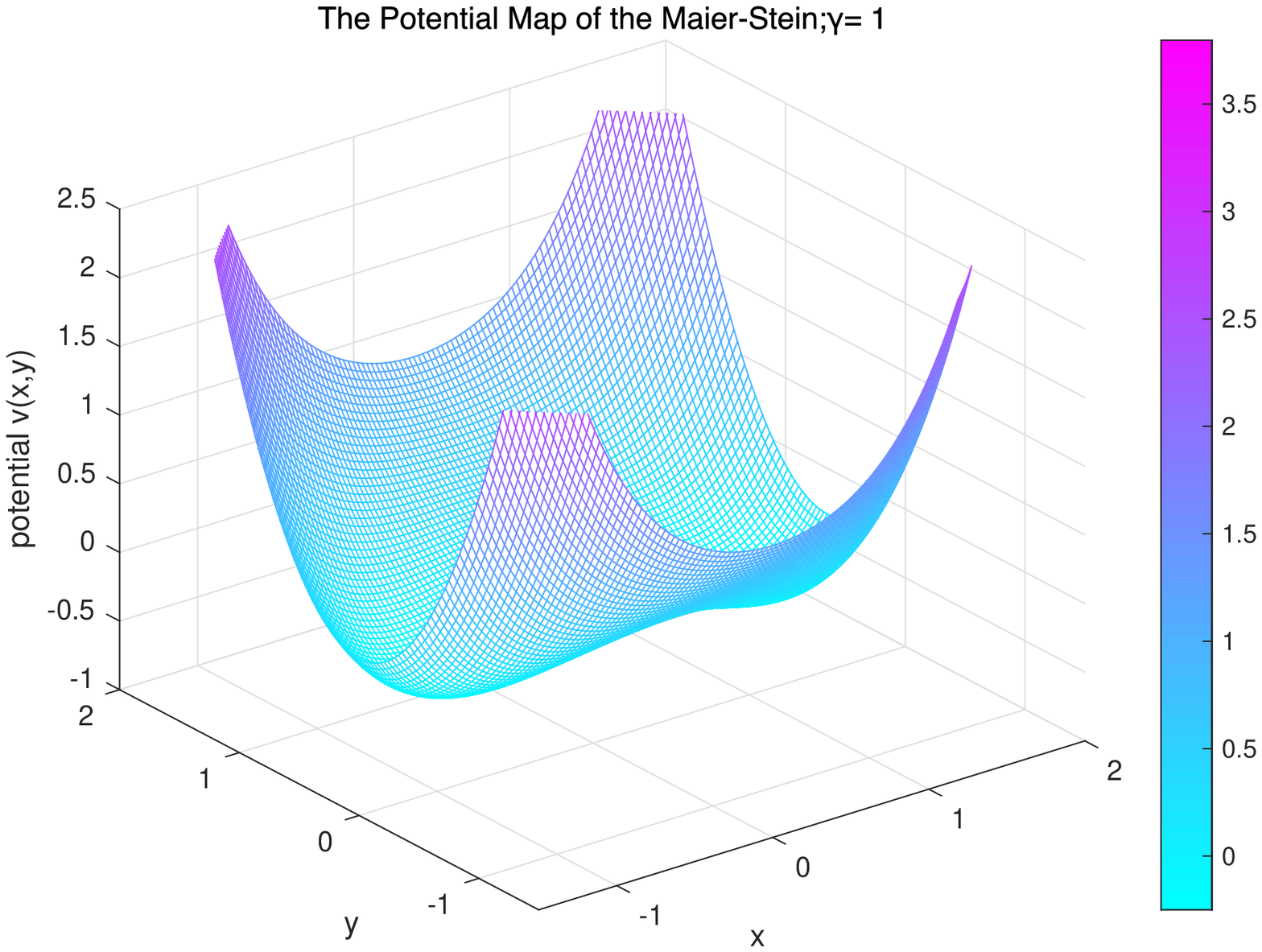}
    \label{fig:potential}
    }
    \caption{The patterns of the Maier-Stein system. The vector field is described by (a), the potential is shown as (b). }
\end{figure}
It is known that $Df$ is symmetric if and only if $\gamma=1$. The domains of attraction are separated by the stable manifold of the unstable saddle point US $(0,0)$, as shown in Fig. 1(a).
In this case, the potential function of the Maier-Stein system is
\begin{equation}
\begin{split}
V(x,y)=-\frac{1}{2}x^2+\frac{1}{4}x^4+\frac{1}{2}y^2+\frac{1}{2}x^2y^2.
\end{split}
\end{equation}

By Theorem 2.5, the Onsager-Machlup action functional for the Maier-Stein system we obtained is
\begin{equation}
\begin{split}
J_0(z,\dot{z})&=-\frac{1}{2}\int_0^1L(z,\dot{z})ds,
\end{split}
\end{equation}
where
\begin{equation}
\begin{split}
L(z,\dot{z})&=|f(z(t))-\dot{z}(t)-\eta|^2+\nabla\cdot f(z(t))\\
&=(x-x^3-\gamma xy^2-\dot{x}-\eta_1)^2-4x^2-\gamma y^2\\
&\ \ \ \ +((1+x^2)y+\dot{y}+\eta_2)^2.
\end{split}
\end{equation}
\begin{figure}[hp]
    \centering
    \includegraphics[scale=0.53]{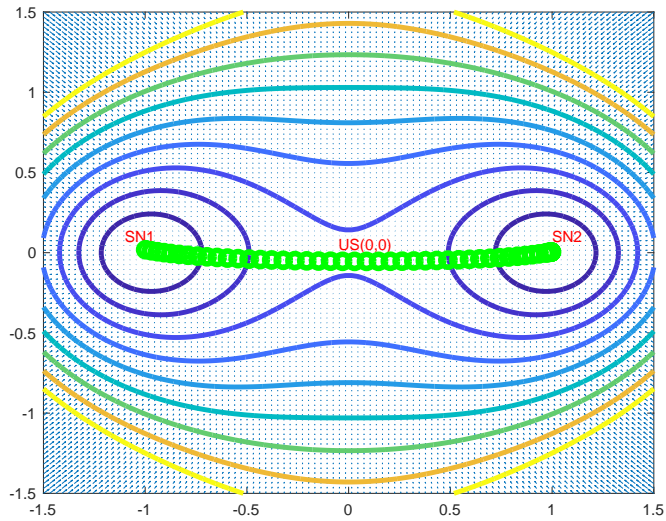}
    \label{fig:path1}
    \caption{The most probable transition pathway from SN2 to SN1 for system (4.1). The small green circle in the figure represents the transition pathway. Initial data z(0)=(1,0) and finial data z(0)=(-1,0). The L\'{e}vy noise $L_1\sim S_{0.5}(1,0.5,0)$ and $L_2\sim S_{0.7}(1,0,0)$.}
\end{figure}

By Theorem 3.1 and Remark 3.3, the corresponding Euler-Lagrange equation is

\begin{equation}
\begin{split}
\ddot{z}=g(z),
\end{split}
\end{equation}
with two boundary conditions $z(0)=(1,0)$ and $z(1)=(-1,0)$, where $g(z)=(g_1(z),g_2(z))$ is given by
\begin{figure}[hp]
    \centering
    \subfloat[The trajectories of x]{
    % \subfigure[]{
    \includegraphics[scale=0.33]{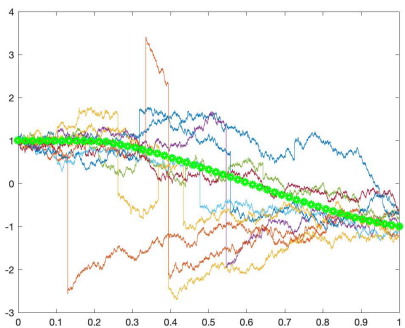}
    \label{fig:X}
    }
    \subfloat[The trajectories of y]{
    \includegraphics[scale=0.33]{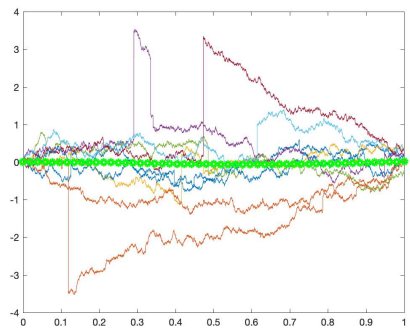}
    \label{fig:Y}
    }
    \caption{The patterns of sample paths and the most probable transition pathway.}
\end{figure}
\begin{equation}
\begin{split}
g_1(x,y)&=3x^5+6x^3y^2-4x^3+3\eta_1x^2+xy^4+2\eta_2xy-3z_1-\eta_1+\eta_1 y^2\\
g_2(x,y)&=3x^4y+2x^2y^3+\eta_2x^2+2\eta_1xy+\eta_2.
\end{split}
\end{equation}

Then we use the shooting neural networks method \cite{ibraheem2011shooting} to solve the two-point boundary value problem (4.7) numerically, which is shown in Fig.2. The green one connecting SN1 and SN2 describes the most probable transition pathway. And the interesting phenomenon is that it chooses the relevantly flat potential areas to transit, which seems somehow reasonable.

More specifically, we investigate the sample paths and the most probable transition pathway for the two components $x$ and $y$ respectively, shown in Fig.3. The most of the sample paths are located around the green one. We emphasise that the most probable transition pathway is not the real sample path of the particles, but it describes the largest probability of sample paths around its neighbourhood. From that, we conclude that it is reasonable to apply the Onsager-Machlup action functional and the Euler–Lagrange equation to determine the most probable transition pathway of the jump-diffusion processes. Please see \cite{transition-path-web} for more details  .
\section{Conclusion and discussion}
In this paper, we have derived the Onsager-Machlup action functional for a class of stochastic dynamical systems with additive L\'{e}vy noise in high dimensions. As the integral of a Lagrangian, the Onsager-Machlup action functional enables the study of the transition phenomena between metastable states. Then by a variational principle, the most probable path connecting two metastable states satisfies the Euler-Lagrange equation. A shooting neural network method works well to solve the Euler-Lagrange equation numerically.

However, for stochastic dynamical system (2.1) with multiplicative diffusion coefficient, the Onsager-Machlup action functional is unknown, even for one dimension. Our method is not applicable here either. Note that for stochastic dynmaical system (2.1) with only multiplicative diffusion coefficient  (i.e., the L\'evy noise is absent), we can use an appropriate reference process whose radial is a Bessel process under the uniform norm. But unfortunately, as long as L\'evy noise is present, it is hard to find such a process whose radial is a process independent of the drift and diffusion coefficients.

In addition, in our method, we use the Brownian motion to absorb the drift by applying the Girsanov theorem. But for stochastic dynamical system (2.1) with pure L\'evy noise (i.e., the Brownian noise is absent), we are unable to  make our method work, because the Girsanov transformation for pure L\'evy noise will result in a martingale, instead of a compensated Poisson process, in a new probability space.

\section*{Acknowledgements}
We would like to thank Jinqiao Duan for helpful comments and show great gratitude to Ying Chao, Yang Li and Xiaoli Chen for helpful discussions. This work was partly supported by NSFC grants 11771449, 11531006 and 11801192.
\bibliographystyle{abbrv}
 %\bibliographystyle{plain}
% \biboptions{square,numbers,sort&compress}
\bibliography{References}
\begin{appendices}

\section{Proof of Theorem 2.5}
\begin{proof}
We will prove this theorem in the following 3 steps.

$\mathbf{Step}$ $\mathbf{1}$: An application of Girsanov transformation.

Let $Y(t)$ be the solution of the stochastic differential equation with respect to
$(\Omega, \mathcal{F}, (\mathcal{F}_t)_{t\geq0}, \mathbb{P})$
\begin{equation}
\begin{split}
dY(t)=Bh(t)dt+BdW(t)+ dL(t), t\in[0,1],\\
\end{split}
\end{equation}
with initial data $Y(0)=x_0\in \mathbb{R}^d$.
By the Girsanov transformation \cite[Theorem 3.17]{jacod2013limit}, $\mathbb{Q}$ defined as $d\mathbb{Q}(\omega)=M_t(\omega)d\mathbb{P}(\omega)$ is a probability measure, where
\begin{align*}
M_t=\exp&\{ \int_0^t\langle B^{-1}(f(s,Y(s))-Bh(s)),dW(s)\rangle\\
\ \ \  &-\frac{1}{2}\int_0^t|B^{-1}(f(s,Y(s))-Bh(s))|^2ds\}.
\end{align*}
The process $\widehat{W}_t:=W_t-\int_0^tB^{-1}(f(s,Y(s))-Bh(s))ds$ is a Brownian motion under the new filtered probability space $(\Omega, \mathcal{F}, (\mathcal{F}_t)_{t\geq0}, \mathbb{Q})$. Moreover, $\widetilde{N}(dt,dy)$ is still the compensated martingale measure with jump measure $\nu$ with respect to $\mathbb{Q}$. Thus the process $Y_t$ respect to $\mathbb{Q}$ has the stochastic differential representation
\begin{equation}
\begin{split}
dY(t)=f(Y(t))dt+Bd\widehat{W}(t)+ dL(t).
\end{split}
\end{equation}
Due to the uniqueness in distribution, we have the equality of $\mu_X$ and $\mu_Y^{\mathbb{Q}}$. Then we obtain
\begin{align*}
\frac{d\mu_X}{d\mu_Y}[Y_t(\omega)]&=\exp\{ \int_0^t\langle B^{-1}(f(Y(s))-Bh(s)),dW(s)\rangle\\
& \ \ \ \ \ -\frac{1}{2}\int_0^t|B^{-1}(f(Y(s))-Bh(s))|^2ds\}.
\end{align*}
Hence,
\begin{equation}
\begin{split}
\mathbb{P}(\|X-\phi^h\|\leq\epsilon)=\mathbb{Q}(\|Y-\phi^h\|\leq\epsilon)=\mathbb{E}[\exp\{\Lambda\}1_{|X^L|\leq\epsilon}],
\end{split}
\end{equation}
where
\begin{equation}
\begin{split}
\Lambda&=\int_0^1\langle B^{-1}(f(X^L(s_{-})+\phi^h(s))-Bh(s)),dW(s)\rangle\\
&\ \ -\frac{1}{2}\int_0^1|B^{-1}(f(X^L(s_{-})+\phi^h(s))-Bh(s))|^2ds.
\end{split}
\end{equation}

$\mathbf{Step}$ $\mathbf{2}$: Representation by a path integral.

We want to study the limiting behaviours of (A.3). Thus, we try to deal with the stochatic integral in (A.4) by a path integral.

By Lemma 2.3, there exists a  $C^2$ function $V:\mathbb{R}^d\rightarrow \mathbb{R}$, such that for each $x$, $DV(x)=(B^{-1})^{*}(B^{-1}F(x)-Bh)$, $D^2V(x)\in L_{(HS)}(\mathbb{R}^d,\mathbb{R}^d)$ and that the mapping $x\mapsto D^2V(x)$ is uniformly continuous on any bounded subset of $\mathbb{R}^d$. Applying It\^{o} formula \cite[Theorem 4.4.7]{applebaum2009levy} for this function $V$ with respect to $Y(t)$, we obtain
\begin{equation}
\begin{split}
\Lambda&=V(y(1))-V(y(0))- \sum\limits_{0\leq t\leq 1}[V(y(t))-V(y(t_{-}))]-\int_0^1b(y(t_{-})) dt\\
&+\int_0^1\int_{|\xi|<1} \langle B^{-1}(f(y(t-))-h(t)),B^{-1}\xi\rangle\nu(d\xi)dt,
\end{split}
\end{equation}
where $b$ is defined as
\begin{equation}
\begin{split}
b(y(t_{-}))&=|B^{-1}(f(y(t_{-}))-Bh(t))|^2+\sum\limits_{j=1}^{d}D^2_{jj}V(y(t_{-})))Be_j\otimes Be_j \\
&+2\langle B^{-1}(f(y(t_{-}))-Bh(t)),h(t)\rangle.
\end{split}
\end{equation}

$\mathbf{Step}$ $\mathbf{3}$: Taylor expansion and asymptotic behavior.

The integrals with respect to time variable $t$ in (A.6) are Riemann integrals. Now we expand the exponent of (A.6) into a Taylor series around $y(t)=\phi^h(t)$ due to $|X^L(t)|\leq\epsilon$, and split the terms of zero order. If we choose $\epsilon$ small enough, the remaining terms can be made arbitrarily small.

Since $B^{-1}F$ is a $C_b^2$ in $x$ uniformly in $t\in[0,1]$, and $V$ is at least $C^2$ function, we obtain
\begin{equation}
\begin{split}
V(y(1))&=V(\phi^h(1))+\langle DV(\phi^h(1)),X^L(1)\rangle+o(\epsilon),\\
b(y(t-))&=|B^{-1}f(\phi^h(t))-h(t)|^2+2\langle B^{-1}f(\phi^h(t))-h(t),h(t))\rangle\\
&\ \ +Tr(D_xf(\phi^h(t)))+Tr(\langle D_x^2f(\phi^h(t)),X^L(t)\rangle)\\
&\ \  +2\langle B^{-1}f(\phi^h(t))-h(t),h(t)\rangle\\
&\ \ +2\langle B^{-1}f(\phi^h(t))-h(t),\langle D_xf(\phi^h(t)),X^L(t)\rangle\rangle +o(\epsilon),\\
\end{split}
\end{equation}
and
\begin{equation}
\begin{split}
\langle B^{-1}f(&y(t-))-h(t),B^{-1}\eta\rangle=\langle
B^{-1}f(\phi^h(t))-h(t),B^{-1}\eta\rangle\\
&\ \ \ \ \ \ \ \ \ \ \ \ \ \ \ \ \ \ \ \ \ \ \ \ \ \ \ \ +\langle B^{-1}\langle D_xf(\phi^h(t)),X^L(t)\rangle,B^{-1}\eta\rangle+o(\epsilon),
\end{split}
\end{equation}
where $\eta=\int_{|\xi|<1}\xi\nu(d\xi)$. Note that
\begin{equation}
\begin{split}
V(\phi^h(1))-V(\phi^h(0))=\int_0^1\langle\dot{\phi}^h(t),(B^{-1})^{*}(B^{-1}f(\phi^h(t))-h(t))\rangle dt.
\end{split}
\end{equation}

Since $|X^A(t)|\leq\epsilon$, by H\"{o}lder inequality and by the fact of $\dot{\phi}^h(t)=h(t)$, we obtain
\begin{equation}
\begin{split}
\Lambda&=V(\phi^h(1))-V(\phi^h(1))\\
&\ \ -\int_0^1Tr(D_xf(\phi^h(t)))dt-\int_0^1|B^{-1}f(\phi^h(t))-h(t)|^2dt\\
&\ \ -2\int_0^1\langle B^{-1}f(\phi^h(t))-h(t),B^{-1}\eta\rangle dt\\
&\ \ +\sum\limits_{0\leq t\leq 1}[V(y(t))-V(y(t-))]+O(\epsilon)\\
&=-\frac{1}{2}\int_0^1|B^{-1}[f(\phi^h(t))-\dot{\phi}^h(t)]|^2dt\\
&\ \ \ -\frac{1}{2}\int_0^1Tr[\nabla_xf(\phi^h(s))]ds\\
&\ \ \ -\frac{1}{2}\int_0^1\langle B^{-1}[f(\phi^h(t))-\dot{\phi}^h(t)],B^{-1}\eta\rangle dt\\
&\ \ \ +\sum\limits_{0\leq t\leq 1}[V(y(t))-V(y(t-))]+O(\epsilon),
\end{split}
\end{equation}
where $\eta=\int_{|\xi|<1}\xi\nu(d\xi)$. As for the remaining term, we have the following control
\begin{equation}
\begin{split}
\sum_{0\leq t\leq 1}[V(y(t))-V(y(t-))]\leq |DV|\sum_{0\leq t\leq 1}|y(t)-y(t-)|,
\end{split}
\end{equation}
where the operator norm $|DV|$ is finite since the operator $D_xV(x)=(B^{-1})^{*}(B^{-1}F(x)-h)$ is $C_b^2$ in $x$. And $y(t)$ is given by
\begin{equation}
\begin{split}
y(t) &= \phi^h(t)+BW(t)+L(t).
\end{split}
\end{equation}
Since $\phi^h(t)+BW(t)$ is continuous in $t\in [0,1]$, we have
\begin{equation}
\begin{split}
\sum_{0\leq t\leq 1}|y(t)-y(t-)|=\sum_{0\leq t\leq 1}|L(t)-L(t-)|.
\end{split}
\end{equation}
Then by Lemma 2.2, we obtain
\begin{equation}
\begin{split}
\sum\limits_{0\leq t\leq 1}[V(y(t))-V(y(t_{-}))]&\leq\|DV\|\sum\limits_{0\leq t\leq 1}|\Delta y(t)|\\
&\leq C_0\sum\limits_{0\leq t\leq 1}\Delta L(t)<\infty
\end{split}
\end{equation}
Combining (A.3), (A.4), (A.10) and (A.14), we have
\begin{equation}
\begin{split}
&\mathbf{P}(\|X-\phi^h\|_2\leq\varepsilon)\propto\mathbf{E}[\exp\{J_0(\phi^h,\dot{\phi}^h)-C+O(\epsilon)\}\mathbf{1}_{\|X^A\|_2\leq\varepsilon}]
\end{split}
\end{equation}
where $C$ is a constant and  $J_0(\phi^h,\dot{\phi}^h)$ is given by
\begin{equation}
\begin{split}
J_0(\phi^h,\dot{\phi}^h)&=-\frac{1}{2}\int_0^1|B^{-1}[(f(\phi^h(t))-\dot{\phi}^h(t)]|^2dt\\
&\ \ \ -\frac{1}{2}\int_0^1Tr[\nabla_xf(\phi^h(t))]dt\\
&\ \ \ -\int_0^1\langle B^{-1}[f(\phi^h(t))-\dot{\phi}^h(t)],B^{-1}\eta\rangle dt.
\end{split}
\end{equation}
Hence, we obtain
\begin{equation}
\begin{split}
\lim\limits_{\epsilon\rightarrow0}\gamma_{\epsilon}(\phi)=\lim\limits_{\epsilon\rightarrow0}\frac{P(\|X-\phi^h\|_2\leq\epsilon)}{P(\|X^L\|_2\leq\epsilon)}\propto \exp\{J_0(\phi^h,\dot{\phi}^h)\}.
\end{split}
\end{equation}
Thus, $J_0(\phi^h,\dot{\phi}^h)$ is the desired Onsager-Machlup action functional by the Definition 2.1. And up to a constant, it can also be written as
\begin{equation}
\begin{split}
J_0(\phi^h,\dot{\phi}^h)&=-\frac{1}{2}\int_0^1|B^{-1}[f(\phi^h(t))-\dot{\phi}^h(t)-\eta]|^2dt\\
&\ \ \ -\frac{1}{2}\int_0^1Tr[\nabla_xf(\phi^h(s))]ds,
\end{split}
\end{equation}
where $\eta=\int_{|\xi|_{\mathbb{R}^d}<1}\xi\nu(d\xi)$. This completes the proof of Theorem 2.5.
\end{proof}
\end{appendices}
\end{document}